\documentclass[oneside,english]{amsart}
\usepackage[T1]{fontenc}
\usepackage[latin9]{inputenc}
\usepackage{amsthm}
\usepackage{amssymb}

\makeatletter
\numberwithin{equation}{section}
\numberwithin{figure}{section}
 \theoremstyle{definition}
 \newtheorem*{defn*}{\protect\definitionname}
\theoremstyle{plain}
\newtheorem{thm}{\protect\theoremname}
  \theoremstyle{plain}
  \newtheorem{prop}[thm]{\protect\propositionname}
  \theoremstyle{remark}
  \newtheorem{rem}[thm]{\protect\remarkname}
  \theoremstyle{definition}
  \newtheorem{example}[thm]{\protect\examplename}
  \theoremstyle{plain}
  \newtheorem{lem}[thm]{\protect\lemmaname}

\usepackage{etoolbox}

\patchcmd{\@maketitle}
  {\ifx\@empty\@dedicatory}
  {\smallskip
    \begin{center}
    \footnotesize
    \begin{tabular}{c}
      Department of Mathematical Sciences\\
      University of Bath, UK   
    \end{tabular}
  \end{center}
  \ifx\@empty\@dedicatory}
  {}{}

\makeatother

\usepackage{babel}
  \providecommand{\definitionname}{Definition}
  \providecommand{\examplename}{Example}
  \providecommand{\lemmaname}{Lemma}
  \providecommand{\propositionname}{Proposition}
  \providecommand{\remarkname}{Remark}
\providecommand{\theoremname}{Theorem}

\begin{document}

\title{Omegas of Agemos in Powerful Groups}

\author{James Williams}
\begin{abstract}
In this note we show that for any powerful $p$-group $G$, the subgroup
$\Omega_{i}(G^{p^{j}})$ is powerfully nilpotent for all $i,j\geq1$
when $p$ is an odd prime, and $i\geq1$, $j\geq2$ when $p=2$. We
provide an example to show why this modification is needed in the
case $p=2$. Furthermore we obtain a bound on the powerful nilpotency
class of $\Omega_{i}(G^{p^{j}})$. 
\end{abstract}

\maketitle

\section*{Introduction}

It is well known that for a powerful $p$-group $G$, the $i$th Agemo
subgroup, $\mho_{i}(G)=G^{p^{i}}$, coincides with the set of $p^{i}$th
powers, and that this subgroup is itself powerful \cite[Corollary 1.2, Proposition 1.7]{Lubotzky1987}.
In \cite{Traustason2018} we introduced the notion of powerful nilpotence,
and showed that for a powerful $p$-group $G$ the groups $G^{p^{i}}$,
for $i\geq1$, are powerfully nilpotent. In some sense dual to the
Agemo subgroups are the Omega subgroups, $\Omega_{i}(G)$. For a powerful
$p$-group $G$ these Omega subgroups are studied in \cite{Fernandez-Alcober2007}. 

In \cite{Traustason2018} we observed how powerfully nilpotent groups
often occur as characteristic subgroups of powerful groups. For example
the proper terms of the derived and lower central series of a powerful
group $G$ are powerfully nilpotent. One aim of this paper is to further
motivate the study of the relationship between powerful groups and
the powerfully nilpotent groups within them, by showing another important
class of characteristic subgroups of powerful groups to be powerfully
nilpotent. 

Furthermore, in \cite{Traustason2018} it was proved that for a powerfully
nilpotent group $G$ of order $p^{n}$, rank $r$, exponent $p^{e}$
and powerful nilpotency class $c$, we have that $e\leq n-c+1$ and
$r\le n-c+1$. We call the quantity $n-c$ the powerful coclass of
$G$. It follows that there are only finitely many groups of any given
powerful coclass, and as such a weak classification of powerfully
nilpotent groups exists in terms of an ``ancestry tree'' \cite{Traustason2018}.
In \cite{article} the question is asked ``which $p$-groups are
subgroups of powerful $p$ -groups?'' and this was studied further
in \cite{Mann2003}. Thus in light of all of this it is interesting
to note which subgroups must be powerfully nilpotent.

By Theorem 1.1 in \cite{GONZALEZSANCHEZ2004193} we know that for a powerful $p$-group $G$, if $N \vartriangleleft G$ and $N \leq G^{p}$ then $N$ is powerful in the case where $p$ is an odd prime. Similarly in the even case if $N \vartriangleleft G$ and $N \leq G^{4}$ then $N$ is powerful.  It follows from this that  $\Omega_{i}(G^{p})$ is powerful for $p$ an odd prime, and in the even case that $\Omega_{i}(G^{4})$ is powerful. However in what follows we give an elementary proof of the fact that these Omega subgroups are powerful. In particular in this note we prove that for an odd prime $p$ and a powerful $p$-group
$G$, the Omega subgroups of any proper Agemo subgroup are powerfully
nilpotent (and hence powerful), and moreover we can obtain a bound
on the powerful nilpotency class. 

 \newtheorem*{thm:resultoddp}{Theorem \ref{thm: main result for p odd}} \begin{thm:resultoddp} Let $G$ be a powerful $p$-group for an odd prime $p$. Then $\Omega_{i}(G^{p^{j}})$ is powerfully nilpotent for $i,j\geq1$. The powerful nilpotency class of $\Omega_{i}(G^{p^{j}})$ is at most $i$.  \end{thm:resultoddp}

We obtain a similar result for $p=2$ with a small modification.

 \newtheorem*{thm:resultevenp}{Theorem \ref{thm:main result for p=00003D2}} \begin{thm:resultevenp} Let $G$ be a powerful $2$ group, then $\Omega_{i}(G^{2^{j}})$ is powerfully nilpotent for all $i\geq1,$ $j\geq2$. Furthermore for $i>1$ the powerful nilpotency class of $\Omega_{i}(G^{2^{j}})$ is at most $i-1$. For $i=1$ the powerful nilpotency class is $1$. \end{thm:resultevenp}

\section*{Preliminaries}

In this section we set up notation and terminology. For a group $G$,
we denote the centre of $G$ by $Z(G)$, the commutator subgroup of
$G$ by $G^{\prime}$, and $G^{n}$ denotes the subgroup generated
by all $n$th powers of elements of $G$. For a $p$-group $G$, the
group $G^{p^{i}}$ is sometimes denoted as $\mho_{i}(G)$ and known
as the $i$th Agemo subgroup of $G$. The $i$th Omega subgroup of
$G$, denoted $\Omega_{i}(G)$, is the subgroup generated by all elements
of $G$ whose order divides $p^{i}$. The exponent of $G$ is denoted
by $\exp G$.
\begin{defn*}
A finite $p$-group $G$ is termed \textbf{powerful} if $p>2$ and
$G^{\prime}\leq G^{p}$, or $p=2$ and $G^{\prime}\leq G^{4}$. 
\end{defn*}

\begin{defn*}
A powerful $p$-group $G$ is said to be \textbf{powerfully nilpotent}
if it has a chain of subgroups $\{1\}=H_{0}\leq H_{1}\leq\dots\leq H_{n}=G$
such that $[H_{i},G]\leq H_{i-1}^{p}$ for $i=1,\dots n$. Such a
chain is called a \textbf{powerfully central chain}. If $G$ is powerfully
nilpotent, then the \textbf{powerful nilpotency class} of $G$ is
the shortest length that a powerfully central chain of $G$ can have. 
\end{defn*}

Notice that the assumption that $G$ is powerful is not needed if
$p$ is odd.

\begin{defn*}
For any prime $p$, a finite $p$-group $G$ is \textbf{strongly powerful}
if $G^{\prime}\leq G^{p^{2}}$. 
\end{defn*}
In \cite{Traustason2018}, it is shown that a strongly powerful group
must be powerfully nilpotent, thus in particular any powerful $2$-group
is powerfully nilpotent. The theory of powerful $p$-groups is developed
in \cite{Dixon2003,Lubotzky1987}. 

For the convenience of the reader we now repeat, without proof, some
results which are used in this paper.

In \cite[Theorem 1]{Fernandez-Alcober2007} the following theorem
is proved. We make extensive use of this theorem in our paper. In
keeping with \cite{Fernandez-Alcober2007}, for a $p$-group $G$
and $x\in G$, we define the meaning of the inequality $o(x)\leq p^{i}$
with $i<0$ to be that $x=1$. Similarly we define $\Omega_{i}(G)=\{1\}$
for $i<0$.
\begin{thm}[Fernández-Alcober]
\label{thm:Gustavos theorem}Let $G$ be a powerful p-group. Then,
for every $i\geq0$:

(i) If $x,y\in G$ and $o(y)\leq p^{i}$, then $o([x,y])\leq p^{i}$.

(ii) If $x,y\in G$ are such that $o(x)\leq p^{i+1}$ and $o(y)\leq p^{i},$
then $o([x^{p^{j}},y^{p^{k}}])\leq p^{i-j-k}$ for all $j,k\geq0$. 

(iii) If $p$ is odd, then $\exp\Omega_{i}(G)\leq p^{i}$. 

(iv) If $p=2$, then $\exp\Omega_{i}(T)\leq2^{i}$ for any subgroup
$T$ of $G$ which is cyclic over $G^{2}$. In particular, $\exp\Omega_{i}(G^{2})\leq2^{i}$. 
\end{thm}
In \cite[Proposition 1.1]{Traustason2018} we prove the following
proposition.
\begin{prop}
\label{prop: G is pn iff G/Gp^2 is pn}Let $G$ be any finite $p$-group
of exponent $p^{e}$ where $e\geq2$. If $G/G^{p^{2}}$ is powerfully
nilpotent, then $G$ is powerfully nilpotent. Furthermore if $G/G^{p^{2}}$
has powerful class $m$, then the powerful class of $G$ is at most
$(e-1)$$m$.\end{prop}
\begin{rem}
\label{rem: lifting up to a pc chain for G from prop 1.1}In the proof
of Proposition \ref{prop: G is pn iff G/Gp^2 is pn}, we show that
if $G/G^{p^{2}}=\bar{H}_{0}\geq\bar{H}_{1}\geq\dots\geq\bar{H}_{m-1}\geq\{\bar{1}\}$
is a powerfully central series, where $\bar{H}_{i}=H_{i}/G^{p^{2}}$,
then the descending chain
\end{rem}
\begin{center}
\begin{align*}
G & = & H_{0} & \geq & H_{1} & \geq & \cdots & \geq & H_{m-1} & \geq & H_{m} & = & G^{p}\\
G^{p} & = & H_{0}^{p} & \geq & H_{1}^{p} & \geq & \cdots & \geq & H_{m-1}^{p} & \geq & H_{m}^{p} & = & G^{p^{2}}\\
\, & \, & \, & \, & \, & \vdots & \, & \vdots & \, & \, & \, & \, & \,\\
G^{p^{e-2}} & = & H_{0}^{p^{e-2}} & \geq & H_{1}^{p^{e-2}} & \geq & \cdots & \geq & H_{m-1}^{p^{e-2}} & \geq & 1
\end{align*}

\par\end{center}

is powerfully central.

\section*{Omega Subgroups of Agemo Subgroups}

The natural place to start when considering Omega subgroups of powerful
$p$-groups is $\Omega_{i}(G)$. However it is not true in general
that $\Omega_{i}(G)$ is powerful and such counter examples are easy
to find. Consider the following example.
\begin{example}
Let $p$ be an odd prime, the $p$-group 
\[
G=\langle a,b,c|a^{p}=b^{p}=c^{p^{2}}=[c,b]=[c,a]=1,[b,a]=c^{p}\rangle
\]
 is powerful (in fact it is powerfully nilpotent), but $\Omega_{1}(G)=\langle a,b,c^{p}\rangle$
is not powerful.
\end{example}
Thus we turn our attention to $\Omega_{i}(G^{p})$. First we shall
use Theorem \ref{thm:Gustavos theorem} to prove that for a powerful
$p$-group $G$, elements in $G^{p}$ of order $p$ commute with each
other and with elements in $G^{p}$ of order $p^{2}$. 
\begin{lem}
\label{lem:elements of order p are central}Let $G$ be a powerful
$p$-group. Let $g_{1},g_{2}\in G^{p}$ where $o(g_{1})=p$ and $o(g_{2})\leq p^{2}$.
Then $[g_{1},g_{2}]=1$ . \end{lem}
\begin{proof}
As $G$ is powerful, we know that elements of $G^{p}$ are $p$th
powers, and so we may assume $g_{1}=a^{p},g_{2}=b^{p}$ for $a,b\in G$
where $o(a)=p^{2}$ and $o(b)\leq p^{3}$ . Using Theorem \ref{thm:Gustavos theorem}(ii)
and taking $x=b$, $y=a$ and $i=2$ we see that $o([x^{p},y^{p}])\leq p^{2-1-1}=1$,
hence $[g_{2},g_{1}]=1$. It follows that the elements in $G^{p}$
of order $p$ commute with the elements in $G^{p}$ of order at most $p^{2}$. 
\end{proof}

Notice that from this we obtain that $\Omega_{1}(G^{p})$ is abelian.
The next result is needed in the proof of Proposition \ref{prop:H is a powerful group},
although it is also of independent interest in the context of better
understanding the relationship between Agemo and Omega subgroups in
powerful $p$-groups.

\begin{prop}
\label{prop:Omega_i(Gp)^j is in omega_=00007Bi-j=00007D}Let $G$
be a powerful $p$-group. Then $\left(\Omega_{i}(G^{p^{k}})\right)^{p^{j}}\!\!\!\leq\Omega_{i-j}(G^{p^{k+j}})$
and $\exp(\Omega_{i}(G^{p^{k}})^{p^{j}})\leq p^{i-j}$ for $i,j\geq0$
and $k\geq1$.\end{prop}
\begin{proof}
Consider an element $x\in\left(\Omega_{i}(G^{p^{k}})\right)^{p^{j}}$.
This element can be written in the form $g_{1}^{p^{j}}\!\cdots g_{t}^{p^{j}}$
where $g_{l}\in\Omega_{i}(G^{p^{k}})$ for each $l\in\{1,\dots,t\}$.
Note that $g_{l}\in G^{p^{k}}$ and so $g_{l}^{p^{j}}\in G^{p^{k+j}}$.
Using Theorem \ref{thm:Gustavos theorem}(iii) if $p$ is odd and
Theorem \ref{thm:Gustavos theorem}(iv) if $p=2$, it follows that
the order of each $g_{l}$ is at most $p^{i}$. Then the order of
each $g_{l}^{p^{j}}$ is at most $p^{i-j}$. Thus each $g_{l}^{p^{j}}\in\Omega_{i-j}(G^{p^{k+j}})$.
As $\Omega_{i-j}(G^{p^{k+j}})$ is a group, it is closed under taking
products and so $x=g_{1}^{p^{j}}\!\cdots g_{t}^{p^{j}}\in\Omega_{i-j}(G^{p^{k+j}})$.
Hence $\left(\Omega_{i}(G^{p^{k}})\right)^{p^{j}}\leq\Omega_{i-j}(G^{p^{k+j}})$.
Then by Theorem \ref{thm:Gustavos theorem}(iii) if $p$ is odd and
Theorem \ref{thm:Gustavos theorem}(iv) if $p=2$, we obtain that
$\exp(\Omega_{i}(G^{p^{k}})^{p^{j}})\leq p^{i-j}$.
\end{proof}

We now consider the case where $p$ is an odd prime. We seek to show
that $\Omega_{i}(G^{p})$ is powerfully nilpotent for all $i\geq1$.
Recall by Proposition \ref{prop: G is pn iff G/Gp^2 is pn} that for
any $p$-group $G$ we have that $G$ is powerfully nilpotent if and
only if $G/G^{p^{2}}$ is powerfully nilpotent. Thus in what follows
we consider $H=\frac{\Omega_{i}(G^{p})}{(\Omega_{i}(G^{p}))^{p^{2}}}$,
for some powerful $p$-group $G$. Let $\mbox{\ensuremath{K=\frac{G}{(\Omega_{i}(G^{p}))^{p^{2}}}}}$.
Notice that $K$ and $K^{p}$ are powerful and that $H\leq K^{p}$. 
\begin{prop}
\label{prop:H is a powerful group}$H$ is a powerful group.\end{prop}
\begin{proof}
The exponent of $H$ is at most $p^{2}$, and so it follows from Lemma
\ref{lem:elements of order p are central} that all elements of order
$p$ are central. We thus only need to consider commutators between
elements of order $p^{2}$. Since $H\leq K^{p}$, we can thus assume
these commutators are of the form $[a^{p},b^{p}]$ where $o(a)=p^{3}=o(b)$.
Applying Theorem \ref{thm:Gustavos theorem}(ii) with $x=a$, $y=b$
and $i=3$ we see that $o([a^{p},b^{p}])\leq p$. Since $K$ is powerful,
we have that $[a^{p},b^{p}]\in[K^{p},K^{p}]=[K,K]^{p^{2}}\leq K^{p^{3}}$,
and hence there exists some $g\in K$ such that $[a^{p},b^{p}]=g^{p^{3}}$,
where $g$ has order at most $p^{4}$. Let $g=x\left(\Omega_{i}(G^{p})^{p^{2}}\right)$.
Then $\mbox{\ensuremath{x^{p^{4}}\in\Omega_{i}(G^{p})^{p^{2}}}}$,
which is of exponent at most $p^{i-2}$, by Proposition \ref{prop:Omega_i(Gp)^j is in omega_=00007Bi-j=00007D}.
Hence $\mbox{\ensuremath{o(x)\leq p^{4+i-2}=p^{i+2}}}$. Then $x^{p^{2}}$
has order at most $p^{i}$ and so $x^{p^{2}}\in\Omega_{i}(G^{p^{2}})$.
Then $g^{p^{3}}=x^{p^{3}}\left(\Omega_{i}(G^{p})^{p^{2}}\right)\in\left(\frac{\Omega_{i}(G^{p^{2}})}{\Omega_{i}(G^{p})^{p^{2}}}\right)^{p}\leq H^{p}$.
Thus $H$ is powerful.\end{proof}
\begin{lem}
\label{lem:H-is-powerfully nilpotent of class at most 2}$H$ is powerfully
nilpotent of powerful nilpotency class at most 2, in particular $H\geq\frac{\Omega_{i}(G^{p^{2}})}{\Omega_{i}(G^{p})^{p^{2}}}\geq1$
is a powerfully central chain.\end{lem}
\begin{proof}
We will show that $H\geq\frac{\Omega_{i}(G^{p^{2}})}{\Omega_{i}(G^{p})^{p^{2}}}\geq1$
is a powerfully central chain. In the proof of Proposition \ref{prop:H is a powerful group}
we saw that $[H,H]\leq\left(\frac{\Omega_{i}(G^{p^{2}})}{\Omega_{i}(G^{p})^{p^{2}}}\right)^{p}$.
We now show that $\frac{\Omega_{i}(G^{p^{2}})}{\Omega_{i}(G^{p})^{p^{2}}}\leq Z(H)$,
to do this we will show that $[\Omega_{i}(G^{p^{2}}),\Omega_{i}(G^{p})]\leq\Omega_{i}(G^{p})^{p^{2}}$.
Consider $[g^{p^{2}},h^{p}]$ for $g,h\in G$ with $o(g)\leq p^{i+2}$
and $o(h)\leq p^{i+1}$. Using Theorem \ref{thm:Gustavos theorem}(ii)
we obtain that $o([g^{p^{2}},h^{p}])\leq p^{i-2}$. As $[g^{p^{2}},h^{p}]\in G^{p^{4}}$
we may write $[g^{p^{2}},h^{p}]=k^{p^{4}}$ for some $k\in G$. Then
$o(k^{p^{2}})\leq p^{i}$ and so $[g^{p^{2}},h^{p}]=(k^{p^{2}})^{p^{2}}\in\Omega_{i}(G^{p})^{p^{2}}$.
Thus $\frac{\Omega_{i}(G^{p^{2}})}{\Omega_{i}(G^{p})^{p^{2}}}\leq Z(H)$
. Hence it follows that $H\geq\frac{\Omega_{i}(G^{p^{2}})}{\Omega_{i}(G^{p})^{p^{2}}}\geq1$
is a powerfully central chain.
\end{proof}

Using Lemma \ref{lem:H-is-powerfully nilpotent of class at most 2}
and Proposition \ref{prop: G is pn iff G/Gp^2 is pn} one can obtain
a powerfully central chain for $\Omega_{i}(G^{p})$ of length $2i-1$.
However a shorter chain is possible. The following Lemma will be used
to reduce the length of the chain.
\begin{lem}
\label{lem:If x is in omega_i(Gp)pj+1 but not in omega_i(Gp2)pj+1 then x is not in G^pj+3 to shorten the chain}$[\Omega_{i}(G^{p^{2}})^{p^{j}},\Omega_{i}(G^{p})]\leq\Omega_{i}(G^{p^{2}})^{p^{j+2}}$
for $i\geq1$ and $j\geq0$.\end{lem}
\begin{proof}
By Proposition \ref{prop:Omega_i(Gp)^j is in omega_=00007Bi-j=00007D}
we know that $[\Omega_{i}(G^{p^{2}})^{p^{j}},\Omega_{i}(G^{p})]\leq[\Omega_{i-j}(G^{p^{2+j}}),\Omega_{i}(G^{p})]$,
hence it suffices to show that $[\Omega_{i-j}(G^{p^{2+j}}),\Omega_{i}(G^{p})]\leq\Omega_{i}(G^{p^{2}})^{p^{j+2}}.$
Consider $g,h\in G$ with $o(g)\leq p^{i+2}$ and $o(h)\leq p^{i+1}$
then $g^{p^{2+j}}\in\Omega_{i-j}(G^{p^{2+j}})$ and $h^{p}\in\Omega_{i}(G^{p})$.
Using Theorem \ref{thm:Gustavos theorem}(ii) we obtain that $o([g^{p^{2+j}},h^{p}])\le p^{i-j-2}$.
Also notice that $[g^{p^{2+j}},h^{p}]\in G^{p^{4+j}}$, and hence
we may write $[g^{p^{2+j}},h^{p}]=k^{p^{4+j}}$ for some $k\in G$,
where $o(k^{p^{4+j}})\leq p^{i-j-2}$. It follows that $k^{p^{2}}\in\Omega_{i}(G^{p^{2}})$
and $[g^{p^{2+j}},h^{p}]=k^{p^{4+j}}\in\Omega_{i}(G^{p^{2}})^{p^{j+2}}$.
Hence $[\Omega_{i-j}(G^{p^{2+j}}),\Omega_{i}(G^{p})]\leq\Omega_{i}(G^{p^{2}})^{p^{j+2}}.$
\end{proof}
Note that if $j>i$ in the above, the inclusion still holds, with
both sides of the inequality being equal to the trivial group.

\begin{thm}
\label{thm:omega(g^p) p odd is pn of class at most 2i-2}If $G$ is
a powerful $p$-group where $p$ is an odd prime, then $\Omega_{i}(G^{p})$
is powerfully nilpotent for all $i\geq1$ and the powerful nilpotency
class of $\Omega_{i}(G^{p})$ is at most $i$.\end{thm}
\begin{proof}
As we observed above, for $i=1$ the group is abelian, thus we may
assume $i\geq2$. Note that if $p^{e}=\exp(\Omega_{i}(G^{p}))<p^{2}$
then by Lemma \ref{lem:elements of order p are central} it follows
the group is abelian and so of powerful class $1$ and so the claim
holds in this case. If $\exp(\Omega_{i}(G^{p}))=p^{2}$ then $H\cong\Omega_{i}(G^{p})$
and so the claim follows by Lemma \ref{lem:H-is-powerfully nilpotent of class at most 2}.
Thus we may assume that $e>2$ and $i\geq2$. In Lemma \ref{lem:H-is-powerfully nilpotent of class at most 2}
we saw that $H=\frac{\Omega_{i}(G^{p})}{(\Omega_{i}(G^{p}))^{p^{2}}}$
has a powerfully central chain $H\geq\frac{\Omega_{i}(G^{p^{2}})}{\Omega_{i}(G^{p})^{p^{2}}}\geq1$.
Then by Remark \ref{rem: lifting up to a pc chain for G from prop 1.1}
we have the following powerfully central chain for $\Omega_{i}(G^{p})$:

\begin{align*}
 & \Omega_{i}(G^{p}) &  & \geq\Omega_{i}(G^{p^{2}}) &  & \geq\Omega_{i}(G^{p})^{p}\\
 & \Omega_{i}(G^{p})^{p} &  & \geq\Omega_{i}(G^{p^{2}})^{p} &  & \geq\Omega_{i}(G^{p})^{p^{2}}\\
 &  &  & \vdots\\
 & \Omega_{i}(G^{p})^{p^{e-2}} &  & \geq\Omega_{i}(G^{p^{2}})^{p^{e-2}} &  & \geq1
\end{align*}

Now using Lemma \ref{lem:If x is in omega_i(Gp)pj+1 but not in omega_i(Gp2)pj+1 then x is not in G^pj+3 to shorten the chain}
we see that the terms $\Omega_{i}(G^{p})^{p^{j}}$ for $j\in\{1,\dots,e-2\}$
are redundant. Noting that by Theorem \ref{thm:Gustavos theorem}(iii)
we have that $\exp\Omega_{i}(G^{p})\leq p^{i}$, we obtain the following
powerfully central chain for $\Omega_{i}(G^{p})$ of length at most
$i$. 

\[
\Omega_{i}(G^{p})\geq\Omega_{i}(G^{p^{2}})\geq\Omega_{i}(G^{p^{2}})^{p}\geq\dots\geq\Omega_{i}(G^{p^{2}})^{p^{i-2}}\geq1.
\]

\end{proof}
Later we shall see an example where this bound is attained. Recall
that for a powerful $p$-group $G$, we have that $G^{p^{j}}$ is
powerful for all $j\geq0$. Given a powerful group $G$, applying
Theorem \ref{thm:omega(g^p) p odd is pn of class at most 2i-2} to
$G^{p^{j}}$ gives that $\Omega_{i}(G^{p^{j+1}})$ is powerfully nilpotent
for all $i\geq1$. Thus we have that for a powerful $p$-group $G$,
where $p$ is an odd prime, all Omega subgroups of the proper Agemo
subgroups are powerfully nilpotent. 
\begin{thm}
\label{thm: main result for p odd}Let $G$ be a powerful $p$-group
for an odd prime $p$. Then $\Omega_{i}(G^{p^{j}})$ is powerfully
nilpotent for $i,j\geq1$. The powerful nilpotency class of $\Omega_{i}(G^{p^{j}})$
is at most $i$.
\end{thm}
We now turn to the case $p=2$. Due to the modification in the definition
of a powerful $2$-group, that is the requirement that $G^{\prime}\leq G^{2^{2}}$,
the arguments used above would require us to show that the group $H$
is abelian. However, this is not true in general. Below we exhibit
an example of a powerful $2$-group such that $\Omega_{2}(G^{2})$
is not powerful, and so we see that Theorem \ref{thm:omega(g^p) p odd is pn of class at most 2i-2}
cannot hold in its current form for $p=2$.
\begin{example}
\label{exa:badly behaved example}Consider the $2-$group 
\[
G=\langle a,b,c|a^{2^{3}}=1,b^{2^{3}}=1,c^{2^{5}}=1,[a,c]=1,[b,c]=1,[a,b]=c^{2^{2}}\rangle.
\]
One can check either by hand or with GAP \cite{GAP4}, that this is
a consistent presentation defining a group of order $2^{11}$. Clearly
$G$ is powerful and so $G^{2}=\langle a^{2},b^{2},c^{2}\rangle$.
Consider $\Omega_{2}(G^{2})$; this subgroup contains everything in
$G^{2}$ of order less than or equal to $4$. In particular it contains
$a^{2}$, $b^{2}$ and $c^{2^{3}}$. Notice $[a^{2},b^{2}]=c^{2^{4}}$.
Hence $\Omega_{2}(G^{2})$ is not abelian, but then it cannot be powerful
for it has exponent at most $4$ (Theorem \ref{thm:Gustavos theorem}(iv))
and any powerful group of exponent at most $4$ is abelian.
\end{example}
Also note that in the example above, the prime $p=2$ can be replaced
with any odd prime $p$ to give a consistent presentation for a powerfully nilpotent
group of order $p^{11}$, where the property still holds that $\Omega_{2}(G^{p})$
is not abelian. Thus in particular $\Omega_{2}(G^{p})$ is not strongly
powerful, yet is still powerfully nilpotent. Thus for $p$ odd we
see that the subgroups $\Omega_{i}(G^{p})$ are an example of characteristic
subgroups of a powerful group $G$ which are powerfully nilpotent
but not necessarily strongly powerful. This is in contrast to the
subgroups $G^{p^{i}}$ for $i\geq1$, and the proper terms of the
derived and lower central series of $G$, which are all strongly powerful
\cite{Traustason2018}. Furthermore observe that $\Omega_{2}(G^{p})$
has powerful nilpotency class $2$ and so the bound from Theorem \ref{thm:omega(g^p) p odd is pn of class at most 2i-2}
is attained. 

For the case $p=2$ we make the following modification - instead of
looking at $\Omega_{i}(G^{p})$ we look at $\Omega_{i}(G^{p^{2}})$. 
\begin{thm}
\label{thm:2 group omega(g^4) is pn}If $G$ is a powerful $2-$group,
then $\Omega_{i}(G^{4})$ is powerfully nilpotent for all $i\geq1$
and furthermore for $i>1$ the powerful nilpotency class of $\Omega_{i}(G^{4})$
is at most $i-1$, for $i=1$ the powerful class is $1$.\end{thm}
\begin{proof}
Consider $\tilde{H}=\frac{\Omega_{i}(G^{4})}{(\Omega_{i}(G^{4}))^{4}}$,
we will show that $\tilde{H}$ is abelian. By Lemma \ref{lem:elements of order p are central}
we only need to consider commutators between elements of order $4$.
Let $\tilde{K}=G/(\Omega_{i}(G^{4}))^{4}$, and notice that $\tilde{K}$
and $\tilde{K}^{4}$ are powerful and $\tilde{H}\leq\tilde{K}^{4}$.
We only need to consider commutators of the form $[a^{4},b^{4}]$
where $o(a)=2^{4}$ and $o(b)=2^{4}$. However then by Theorem \ref{thm:Gustavos theorem}(ii),
setting $i=4$ yields that $o([a^{2^{2}},b^{2^{2}}])\leq p^{4-2-2}$
and thus the commutator is trivial. It follows that $\tilde{H}$ is
abelian. Suppose that $\exp(\Omega_{i}(G^{4}))=p^{e}$. If $e=1$
then $\Omega_{i}(G^{4})$ is abelian and so of powerful nilpotency
class $1$, otherwise by Proposition \ref{prop: G is pn iff G/Gp^2 is pn}
the powerful class of $\Omega_{i}(G^{4})$ is at most $e-1$. Since
$\Omega_{i}(G^{p^{2}})\leq\Omega_{i}(G^{p})$, by Theorem \ref{thm:Gustavos theorem}(iv)
we obtain that $e\leq i$ and so the result follows.
\end{proof}
As in the odd case, we can apply the above result to $G^{2^{j}}$
to obtain the following.
\begin{thm}
\label{thm:main result for p=00003D2}Let $G$ be a powerful $2$
group, then $\Omega_{i}(G^{2^{j}})$ is powerfully nilpotent for all
$i\geq1,$ $j\geq2$. Furthermore for $i>1$ the powerful nilpotency
class of $\Omega_{i}(G^{2^{j}})$ is at most $i-1$. For $i=1$ the
powerful nilpotency class is $1$.
\end{thm}

\section*{Acknowledgments}

I would like to thank Dr Gunnar Traustason and Dr Gareth Tracey for
their advice and encouragement with this paper. I am also thankful
for the suggestions of an anonymous referee leading to improved bounds
in Theorem \ref{thm:omega(g^p) p odd is pn of class at most 2i-2}.
I am grateful to the EPSRC for their financial support (grant number
1652316).

\bibliographystyle{plain}
\bibliography{referenceDatabase}

\end{document}